\documentclass[12pt, a4paper]{amsart}

\usepackage{a4}
\usepackage{amssymb}
\usepackage{amsmath}
\usepackage{amsthm}
\usepackage{amstext}
\usepackage{amscd}
\usepackage{latexsym}
\usepackage{graphics}
\usepackage{color}
\usepackage{enumitem}
\usepackage{bbm}
\usepackage[all]{xy}

\usepackage[colorlinks, pagebackref]{hyperref}
\hypersetup{
  colorlinks=true,
  citecolor=blue,
  linkcolor=blue,
  urlcolor=blue}

\theoremstyle{plain}

\newtheorem*{rost}{Rost Nilpotence Principle}

\newtheorem{theorem}{Theorem}[section]

\newtheorem{lemma}[theorem]{Lemma}

\theoremstyle{definition}

\newtheorem{remark}[theorem]{Remark}

\numberwithin{equation}{section}

\newcommand{\SmProj}{{\rm SmProj}}
\newcommand{\Corr}{{\rm Corr}}
\newcommand{\Chow}{{\rm Chow}}

\newcommand{\CH}{{\rm CH}}

\newcommand{\res}{{\rm res}}

\newcommand{\Dim}{{\rm dim \ }}

\newcommand{\Hom}{{\rm Hom}}
\newcommand{\End}{{\rm End}}

\newcommand{\Spec}{{\rm Spec \,}}

\newcommand{\Char}{{\rm char \,}}

\newcommand{\Gal}{{\rm Gal}}

\newcommand{\sF}{{\mathcal F}}

\newcommand{\sH}{{\mathcal H}}

\newcommand{\sK}{{\mathcal K}}

\newcommand{\G}{{\mathbb G}}
\newcommand{\HH}{{\mathbb H}}

\renewcommand{\P}{{\mathbb P}}
\newcommand{\Q}{{\mathbb Q}}

\newcommand{\Z}{{\mathbb Z}}

\def\<{\langle}
\def\>{\rangle} 
\def\-{\overline} 
\def\~{\widetilde}
\def\^{\widehat}
\mathchardef\mhyphen="2D

\input{xy}
\xyoption{all}

\begin{document}

\title{Rost nilpotence and \'etale motivic cohomology}

\author{Andreas Rosenschon}
\address{Mathematisches Institut, Ludwig-Maximilians Universit\"at, Theresienstr. 39, 
D-80333 M\"unchen, Germany.}
\email{axr@math.lmu.de}

\author{Anand Sawant}
\address{Mathematisches Institut, Ludwig-Maximilians Universit\"at, Theresienstr. 39, 
D-80333 M\"unchen, Germany.}
\email{sawant@math.lmu.de}
\date{\today}


\begin{abstract}
A smooth projective scheme $X$ over a field $k$ is said to satisfy the Rost nilpotence principle if 
any endomorphism of $X$ in the category of Chow motives that vanishes on an extension of the base 
field $k$ is nilpotent.  We show that an \'etale motivic analogue of the Rost nilpotence principle 
holds for all smooth projective schemes over a perfect field.  This provides a new
approach to the question of Rost nilpotence and allows us to obtain an elegant proof of Rost nilpotence for surfaces, as well as for birationally ruled threefolds over a field of characteristic $0$.
\end{abstract}

\date{}
\maketitle

\section{Introduction}
\label{section introduction}

Let $X$ be a smooth projective scheme over a field $k$, and let $\Chow(k)$ denote the category of Chow motives over $k$, see \cite{Manin}, \cite{Scholl}, for instance.  We say that the \emph{Rost nilpotence principle} holds for $X$ if for any field extension $E/k$, the kernel of the homomorphism $\End_{\Chow(k)}(X) \to \End_{\Chow(E)}(X_E)$ consists of nilpotent elements. This was first proved by Rost for any smooth projective quadric over a field \cite{Rost-lemma}; 
it follows from this result that there is a decomposition of the Chow motive of a quadric into simpler motives, which is an essential tool in Voevodsky's proof of the Milnor conjecture \cite{Voevodsky-Milnor-conjecture}.  Chernousov, Gille and Merkurjev \cite{Chernousov-Gille-Merkurjev} proved that the Rost nilpotence principle holds for isotropic projective homogeneous varieties for a semisimple algebraic group.  Later, Gille showed that the Rost nilpotence principle holds for geometrically rational surfaces (in arbitrary characteristic) \cite{Gille-geometrically-rational-surfaces} \cite{Gille-Surfaces}, and for smooth, projective, geometrically integral surfaces (in characteristic 0) \cite{Gille-Surfaces}. The Rost nilpotence principle has proved to be very useful in the study of motivic decompositions and is expected to hold for all smooth projective schemes. 

In order to prove that Rost nilpotence holds for a smooth projective scheme $X$ over a field $k$ of characteristic $0$, it suffices to show that for a finite Galois field extension $E/k$ the kernel of the restriction map $\End_{\Chow(k)}(X) \to \End_{\Chow(E)}(X_E)$ consists of nilpotent elements. The approach by Gille to prove this statement for surfaces uses two nontrivial results.  The first input is a result of Rost \cite[Proposition 1]{Rost-lemma}, originally proved using Rost's fibration spectral sequence for cycle modules (see also \cite{Brosnan} for a purely intersection-theoretic proof).  The second input is a Galois cohomological description of codimension $2$ cycles on a scheme that are annihilated after base change to $E$, obtained by Colliot-Th\'el\`ene and Raskind \cite[Proposition 3.6]{Colliot-Thelene-Raskind}, building on ideas of Bloch used in the study of codimension $2$ cycles on rational surfaces \cite{Bloch rational surfaces}.  More precisely, the essential tool is the following vanishing result of Colliot-Th\'el\`ene \cite[Theorem 1, Remark 5.2]{Colliot-Thelene} and Suslin \cite[Theorem 5.8]{Suslin}: $H^1(\Gal(E/k),K_2E(X)/K_2E)=0$, where $X$ is a geometrically irreducible variety with a $k$-rational point. We remark that the analogue of this vanishing result for higher Galois cohomology groups does not hold, thus Gille's proof does not generalize to higher dimensions. 

We consider the \'etale motivic or Lichtenbaum cohomology groups 
 defined as the hypercohomology groups of Bloch's cycles complex $\Z(n)_{\text{\'et}}$, considered as a complex of \'etale sheaves; for details, see Section \ref{section etale motivic cohomology}. A correspondence $\alpha$ in $\End_{\Chow(k)}(X)$ acts on the Lichtenbaum cohomology groups of $X \times_k X$, after inverting the exponential characteristic of the base field. Thus an evident \'etale motivic analogue of the Rost nilpotence principle is to ask whether this action in the \'etale setting is nilpotent. We show that Rost nilpotence holds in this setting in arbitrary dimension in the following sense: 

\begin{theorem}
\label{maintheorem}
Let $X$ be a smooth projective scheme over a perfect field $k$.  Let $\alpha \in \End_{\Chow(k)}(X)$ be a correspondence such that for a Galois field extension $E$ of $k$ the image $\alpha_E\in \End_{\Chow(E)}(X_E)$ is trivial. Then the action of $\alpha$ on the Lichtenbaum cohomology groups of $X \times_k X$ is nilpotent, after inverting the exponential characteristic. 
\end{theorem}

We remark that analogous to the category $\Chow(k)$ of Chow motives over a field $k$, one can construct a category $\Chow_{\rm L}(k)$ of \'etale motivic or Lichtenbaum Chow motives.  In particular, the proof of Theorem \ref{maintheorem} shows that the analogue of Rost nilpotence in the \'etale motivic or Lichtenbaum setting holds for schemes of arbitrary dimension, provided the underlying field $k$ is perfect.

There is canonical map from Chow groups to Lichtenbaum Chow groups, which allows us to compare Rost nilpotence in the usual sense with the \'etale motivic variant proved in Theorem \ref{maintheorem}.  Using the Bloch-Kato conjecture, proved by Rost-Voevodsky \cite{Voevodsky-Milnor-conjecture}, \cite{Voevodsky-Bloch-Kato}, the kernel of this comparison map can be identified with the quotient of a group, which can be computed via a spectral sequence in terms of cohomology groups of certain well-studied sheaves. We analyze the action of a correspondence which is annihilated by a Galois field extension of the base field on the kernel of the comparison map.  In case of dimension $\leq 2$ over a field of characteristic $0$, our approach yields an elegant proof of the Rost nilpotence principle, see Theorem \ref{theorem Rost nilpotence dim at most 2}.  We note that even for surfaces this generalizes the result of Gille \cite{Gille-Surfaces}, since we do not have to impose the condition of geometric integrality.  Moreover, we can improve the bound on the nilpotence exponent 
obtained by Gille, see Remark \ref{remark Gille comparison}.  We also show that the Rost nilpotence principle holds for birationally ruled threefolds over a field of characteristic $0$.  We note that in this case the restriction on the characteristic is also needed because of use of the weak factorization theorem \cite{AKMW} in the proof.

\subsection*{Acknowledgements}
We are grateful to Najmuddin Fakhruddin for pointing out Lemma \ref{lemma threefolds birational invariant} and the use of the weak factorization theorem.  
We also thank Fr\'ed\'eric D\'eglise for discussions and Alexander Merkujev and Burt Totaro for comments and remarks on earlier versions of this paper.  Finally, we thank the referee(s) for a careful reading of the paper and helpful comments that lead to an improved exposition.  This research was partially funded by the DFG. 

\subsection*{Notation.}

Let $X$ be a scheme over a field $k$; we assume $X$ to be separated, of finite type and 
equidimensional.  By a surface we mean a scheme of dimension $2$ and by a threefold we mean a scheme of dimension $3$.  We write $X^{(i)}$ for the points of codimension $i$, and $\CH_i(X)$ 
(resp. $\CH^i(X)$) for the Chow group of algebraic cycles of dimension $i$ (resp.~codimension 
$i$) on $X$ modulo rational equivalence \cite{Fulton}. In particular, if the $k$-scheme $X$ 
is of dimension $d$ over $k$, we have $\CH_i(X) = \CH^{d-i}(X)$.

If $E/k$ is a field extension, we set $X_E=X \times_{\Spec k} \Spec E$; if $\-k$ is an algebraic 
closure of $k$, then $\-X=X\times_{\Spec k}\Spec{\-k}$. If $f:X \to Y$ is a morphism of 
schemes over $k$, its pullback along $\Spec E \to \Spec k$ will be denoted by 
$f_E: X_E \to Y_E$.  For an integral $k$-scheme $X$, we write $k(X)$ for its 
function field and $E(X)$ for the function field of $X_E$. If $x$ is a point of $X$, then 
$k(x)$ will denote its residue field.  

\section{Correspondences, Chow motives and Rost nilpotence}
\label{section correspondences}

In this section, we set up the notation for the article and give a precise statement of the Rost nilpotence principle.  For more details on the basic properties of Chow motives and their relationship with motivic cohomology, we refer the reader to \cite{Manin},\cite{Scholl}, \cite[Chapter 5]{Voevodsky-Suslin-Friedlander}.  For cycle modules and their basic properties, we refer the reader to \cite{Rost} (see also \cite{Elman-Karpenko-Merkurjev}).

Let $\SmProj/k$ be the category of smooth projective schemes over a field $k$. The category $\Corr^0(k)$ of correspondences of degree $0$ over $k$ has the same objects as $\SmProj/k$ and as morphisms
\[
{\Hom}_{\Corr^0(k)}(X,Y)= \underset{j=1}{\overset{r}{\oplus}} \CH^{\Dim Y} (X_j\times Y),
\]
\noindent where $X_1, \ldots, X_r$ are the irreducible components of $X$.  If $f \in {\Corr^0_k}(X,Y)$ and $g \in {\Corr^0_k}(Y,Z)$, then their composition $g \circ f \in {\Corr^{0}_k}(X,Z)$ is defined by the formula
\[
g \circ f = {p_{XZ}}_*(p_{XY}^*(f) \cdot p_{YZ}^*(g)), 
\]
\noindent where $p_{XY}$, $p_{YZ}$ and $p_{XZ}$ are the projection maps from $X \times_k Y \times_k Z$ to $X \times_k Y$, $Y \times_k Z$ and $X \times_k Z$, respectively, and $\cdot$ is 
the intersection product.   The category $\Chow^{\rm eff}(k)$ of effective Chow motives is the idempotent completion of $\Corr^0(k)$.  We will denote the category of Chow motives by $\Chow(k)$ and by $h: \SmProj/k \to \Chow(k)$ the canonical functor that associates with a smooth scheme its Chow motive.  For an object $M$ of $\Chow(k)$ and $i \in \mathbb Z$, we will denote by $M(i)$ its $i$th Tate twist.  To simpilfy notation, we will henceforth write 
\[
\begin{array}{rcl}
\Hom_k(X,Y) & = & \Hom_{\Chow(k)}(h(X),h(Y)),
\vspace{0.1cm}
\\ 
\End_k(X) & = & \End_{\Chow(k)}(h(X)).
\end{array}
\]

We will be interested in the action of $\End_k(X)$ (correspondences of degree $0$ from $X$ 
to itself) on the Chow groups of the self-product $X \times_k X$; this action is simply given 
by composition of correspondences
\begin{equation}
\label{equation action of End}
\End_k(X) \times \CH^n(X \times_k X) \to \CH^n(X \times_k X); \ (\alpha, \beta) \mapsto \alpha \circ \beta.
\end{equation}

Let $E/k$ be a field extension. Then $X\mapsto X_E$ induces a restriction functor
\[
\res_{E/k}: \Chow(k) \to \Chow(E).  
\]
If $M$ is an object (resp. $f$ is a morphism) in $\Chow(k)$, we write $M_{E}$ (resp. $f_E$) for its 
image under the restriction map $\res_{E/k}(M)$.  With this setup, we can state the Rost Nilpotence Principle:  

\begin{rost}
Let $X$ be a smooth projective scheme over a field $k$. Then the Rost nilpotence principle holds 
for $X$ if for every $\alpha \in \End_k(X)$ such that $\alpha_E=0$ for some field extension $E/k$, 
there exists an integer $N $ (possibly depending on $\alpha$) such that $\alpha^{\circ N} =0$,
i.e. $\alpha$ is nilpotent as a correspondence.
\end{rost}

\section{\'Etale motivic cohomology and actions of correspondences}
\label{section etale motivic cohomology}

In this section, we recall the definition and basic properties of \'etale motivic or Lichtenbaum cohomology and give the proof of Theorem \ref{maintheorem}.  Throughout this section, we will denote by $p$ the exponential characteristic of $k$.

Let $X$ be a smooth scheme over $k$ and let $z^n(X,\bullet)$ be the cycle complex 
defined by Bloch \cite{Bloch Higher Chow groups} whose homology groups define the higher
Chow groups
\[
\CH^n(X, m) = H_m(z^n(X,\bullet)). 
\]
\noindent The presheaf $z^n(-,\bullet): U\mapsto z^n(U,\bullet)$ is a sheaf on the 
(small) \'etale site (see \cite[Section 2.2]{Geisser-Levine-Crelle}), and therefore defines a complex of sheaves in the \'etale and Zariski topology.  It is shown in \cite{Bloch Higher Chow groups} that the cycle complex is covariantly functorial for proper maps and contravariantly functorial for arbitrary maps of smooth schemes over a field (the latter assertion requires a moving lemma, which is proved in \cite{Bloch-Moving}).

If $A$ is an abelian group, we have the complex 
$A_X(n)=(z^n(-,\bullet) \otimes A)[-2n]$ of Zariski sheaves (on $X_{\text{Zar}}$) and the 
analogous complex of $A_X(n)_{\text{\'et}}$ of \'etale sheaves (on $X_{\text{\'et}}$). 
The motivic and \'etale motivic or Lichtenbaum cohomology groups with coefficients in $A$ 
are defined as the hypercohomology groups of these complexes
\[
\begin{split}
H_{\rm M}^{m}(X, A(n)) & = \mathbb H^m_{\text{Zar}}(X, A_X(n)), \\
H_{\rm L}^{m}(X, A(n)) & = \mathbb H^m_{\text{\'et}}(X, A_X(n)_{\text{\'et}}). 
\end{split}
\]
With this definition, one has $H_{\rm M}^{m}(X, \Z(n))=\CH^n(X, 2n-m)$ for all $m,n$; 
in particular, if $m=2n$, then $H^{2n}_{\rm M}(X,\Z(n))=\CH^n(X)$ is the usual Chow group.

Analogously, one defines for $m=2n$ the Lichtenbaum Chow groups by
\[
\CH_{\rm L}^n(X)= H^{2n}_{\rm L}(X,\Z(n)), 
\]
\noindent and more generally for $m\geq 0$, the higher Lichtenbaum Chow groups by
\[
\CH_{\rm L}^n(X,m)= H^{2n-m}_{\rm L}(X,\Z(n)). 
\]
Note that $\CH_{\rm L}^n(X,m) = 0$ for $n<0$, because $\Z(n)_{\text{\'et}}$ 
is trivial for $n<0$.  If $\pi: X_{\text{\'et}} \to X_{\rm Zar}$ denotes the canonical morphism of 
sites, then the associated adjunction $\Z_X(n) \to R\pi_* \pi^* \Z_X(n) = R \pi_* \Z_X(n)_{\text{\'et}}$ 
induces comparison (or cycle class) maps
\begin{equation}
\label{equation comparison map}
\CH^n(X, m) \xrightarrow{\gamma} \CH^n_{\rm L}(X, m), 
\end{equation}
for all $m, n$.  With rational coefficients, the adjunction $\Q_X(n) \to R\pi_* \pi^* \Q_X(n)$ is an 
isomorphism (see \cite[Th\'eor\`eme 2.6]{Kahn}, for example). Thus, rationally, we have 
\begin{equation}
\label{equation rational identification}
\CH^n(X, m) \otimes \Q \cong \CH^n_{\rm L}(X, m) \otimes \Q,
\end{equation}
for all $m,n$.  Geisser-Levine have shown in \cite[Theorem 8.5]{Geisser-Levine-Invent} and 
\cite[Theorem 1.5]{Geisser-Levine-Crelle} that if $\ell$ is a prime and $r$ is a positive integer, 
one has on $X_{\text{\'et}}$ the quasi-isomorphisms
\[
(\Z/ \ell^r \Z)_X(n) \xrightarrow{\sim}
\begin{cases}
\mu_{\ell^r}^{\otimes n}, \quad \quad \quad \quad \text{ if $\ell \neq \Char(k)$;} \\
\nu_r(n)[-n],~ ~ \quad \text{ if $\ell = \Char(k)$,}
\end{cases}
\]
\noindent where $\nu_r(n)$ is the $n$-th logarithmic de Rham-Witt sheaf \cite{Milne logarithmic}, \cite{Illusie}. 
Let
\[
(\Q/\Z)_X(n) = \underset{\ell}{\oplus}~ \Q_{\ell}/\Z_{\ell}(n), 
\]
\noindent where $\ell$ runs through all primes, and where
\[
\Q_{\ell}/\Z_{\ell}(n)=
\begin{cases}
\underset{r}{\varinjlim}~ \mu_{\ell^r}^{\otimes n}, \quad \quad \quad \quad \text{ if $\ell \neq \Char(k)$;} \\
\underset{r}{\varinjlim}~ \nu_r(n)[-n],~ ~ \quad \text{ if $\ell = \Char(k)$.}
\end{cases}
\]
\noindent Therefore, with divisible coefficients Lichtenbaum and \'etale cohomology coincide
\begin{equation}
\label{equation identification with etale cohomology}
H_{\rm L}^{m}(X, \Q/\Z(n)) = H^m_{\text{\'et}}(X, (\Q/\Z)_X(n)).
\end{equation}

There are product maps on motivic cohomology (which are induced from the usual external product of cycles at the level of cycle complexes followed by pullback along the diagonal) 
\[
H^m_{\rm M}(X,R(n))\otimes H^{m'}_{\rm M}(X,R(n'))\rightarrow H^{m+m'}_{\rm M}(X,R(n+n')),
\]
and similar product maps for the Lichtenbaum cohomology groups with coefficients in any commutative ring $R$.  Both motivic and Lichtenbaum cohomology groups are contravariantly functorial for arbitrary morphisms between smooth schemes.  In order to get an action of correspondences on Lichtenbaum cohomology groups by a formula analogous to \eqref{equation action of End}, we need appropriate covariant functoriality of Lichtenbaum cohomology groups, which we briefly describe below, using comparison with extension groups in the triangulated category ${\rm DM}_{\text{\'et}}(k, R)$ of \'etale motives (see \cite[Chapter 5]{Voevodsky-Suslin-Friedlander} or \cite{Cisinski-Deglise}).  We will use the notation and terminology of \cite{Cisinski-Deglise}.  It has been shown in \cite[Section 7.1]{Cisinski-Deglise} that there is a canonical map
\[
\rho^{m,n}_X: H^m_{\rm L}(X, \Z(n)) \to H^m_{\text{\'et}}(X, \Z(n)),
\]
where $H^m_{\text{\'et}}(X, \Z(n))$ is defined to be the group $\Hom_{{\rm DM}_{\text{\'et}}(k, R)} (\Z(X),\Z(n)[m])$.  By \cite[Theorem 7.1.2]{Cisinski-Deglise}, $\rho^{m,n}_X$ becomes an isomorphism after tensoring with $\Z[1/p]$.  By \cite[Corollary 6.2.4]{Deglise-RR}, any projective morphism $f: X \to Y$ of relative dimension $r$ between smooth schemes induces Gysin/pushforward morphisms 
\[
f_*: H^m_{\text{\'et}}(X, R(n)) \to H^{m-2r}_{\text{\'et}}(Y, R(n-r))
\]
satisfying the projection formula.  Moreover, the cycle class map 
\[
\sigma: H^m_{\rm M}(X, R(n)) \xrightarrow{\gamma} H^m_{\rm L}(X, R(n)) \xrightarrow{\rho^{m,n}_X} H^{m}_{\text{\'et}}(X, R(n)) 
\]
is compatible with pushforwards with respect to projective maps between regular schemes, where on the left-hand
side one considers the usual pushforwards on Chow groups and on the right-hand side the Gysin morphisms of \cite{Deglise-RR} (see \cite[Remark 7.1.12]{Cisinski-Deglise}).  One therefore gets an action of $\End_k(X)$ by the formula analogous to \eqref{equation action of End} on the groups $H^m_{\text{\'et}}(X \times_k X, R(n))$: 

\begin{equation}
\label{End-action}
\begin{split}
\End_k(X) \times H^{m}_{\text{\'et}}(X \times_k X, R(n)) &\to H^{m}_{\text{\'et}}(X \times_k X, R(n)); \\
(\alpha, \beta) &\mapsto {p_{13}}_*(p_{12}^*\sigma(\alpha) \cdot p_{23}^*\beta),  
\end{split}
\end{equation}
where $p_{ij}$ denotes the projection map $X \times_k X \times_k X \to X \times_k X$ on the $i,j$th components.  Consequently, one gets an action of $\End_k(X)$ on the Lichtenbaum cohomology groups $H^m_{\rm L}(X \times_k X, R(n))$ after inverting the exponential characteristic (that is, after tensoring with $\Z[1/p]$).

\begin{remark}
In fact, with the terminology of \cite{Cisinski-Deglise}, the category ${\rm DM}_{\text{\'et}}(k, R)$ satifies the Grothendieck six functor formalism along with absolute purity and duality properties (see \cite[Corollary 5.5.5, Theorem 5.6.2 and Theorem 6.2.17]{Cisinski-Deglise}). 
\end{remark}

The main ingredient in the proof of Theorem \ref{maintheorem} is the Hochschild-Serre spectral sequence 
\begin{equation}
\label{equation motivic Hochschild-Serre}
E_2^{r,s} = H^r (G, H^{s}_{\rm L} (Y_E, \Z(n))) \Rightarrow H^{r+s}_{\rm L} (Y, \Z(n))
\end{equation}
for \'etale motivic cohomology (see, for example, \cite[page 31]{CT-Kahn}), where $G$ denotes the Galois group of the Galois field extension $E/k$.  For the convenience of the reader, we briefly recall its construction and justify the convergence.  Let ${\rm Shv}(Y_{\text{\'et}})$ denote the category of \'etale sheaves on $Y$ and let $G\mhyphen{\rm Mod}$ denote the category of the category of $\Z[G]$-modules.  Given a cochain complex $C^{\bullet}$ of \'etale sheaves on $Y$, one has the hypercohomology spectral sequence \cite[5.7.9]{Weibel-Hbook} associated with the functor ${\rm Shv}(Y_{\text{\'et}}) \to G\mhyphen{\rm Mod}$ defined by $\sF \mapsto \sF(Y_E)^G$:
\[
E_2^{r,s} = H^r (G, \HH^{s}_{\text{\'et}} C^{\bullet}(Y_E)) \Rightarrow \HH^{r+s}_{\text{\'et}} C^{\bullet}(Y),
\]
which converges if the complex $C^{\bullet}$ is bounded or cohomologically bounded (see \cite[2C]{Kahn} and the references cited there).  The spectral sequence \eqref{equation motivic Hochschild-Serre} can be seen as the hypercohomology spectral sequence of the complex $\Z_Y(n)_{\text{\'et}}$.  Note that \'etale hypercohomology of the complex $\Z_Y(n)_{\text{\'et}}$ is Zariski hypercohomology of the complex $R\pi_*\Z_Y(n)_{\text{\'et}}$ of Zariski sheaves on $Y$.  Consider the exact triangle
\[
R\pi_*\Z_Y(n)_{\text{\'et}} \to R\pi_*\Q_Y(n)_{\text{\'et}} \to R\pi_*(\Q/\Z)_Y(n)_{\text{\'et}} \to R\pi_*\Z_Y(n)_{\text{\'et}}[1].
\]
The hypercohomology spectral sequence for $R\pi_*\Q_Y(n)_{\text{\'et}}$ converges since it is cohomologically bounded (since $\Q_Y(n) \simeq R\pi_* \pi^* \Q_Y(n) = R\pi_*\Q_Y(n)_{\text{\'et}}$ and since $Y$ has finite Zariski cohomological dimension). The hypercohomology spectral sequence for $(R\pi_*\Q/\Z)_Y(n)_{\text{\'et}}$ converges, the complex being bounded.  Consequently, the spectral sequence \eqref{equation motivic Hochschild-Serre} converges.

We are now set to give a proof of Theorem \ref{maintheorem}.  We will treat the cases when the base field $k$ is of characteristic $0$ and positive characteristic separately.

\subsection*{Proof of Theorem \ref{maintheorem}}

\begin{proof}[Proof in the case $\Char(k)=0$]
\hspace{1cm} 

We may assume that $E$ is a finite Galois field extension of $k$ with $G = \Gal(E/k)$.  Consider the Hochschild-Serre spectral sequence
\begin{equation}
\label{motivic Hochschild-Serre}
E_2^{r,s} = H^r (G, H^{s}_{\rm L} (Y_E, \Z(n))) \Rightarrow H^{r+s}_{\rm L} (Y, \Z(n)).
\end{equation}
Clearly, we have $E_2^{r,s} = 0$, if $r<0$.  From \eqref{equation identification with etale cohomology}, it follows that $H_{\rm L}^{s}(Y_E, \Q/\Z(n)) = H^s_{\text{\'et}}(Y_E, (\Q/\Z)_Y(n)) = 0$ for $s<0$.  Hence, the long exact sequence of hypercohomology associated to the exact triangle
\[
\Z_Y(n)_{\text{\'et}} \to \Q_Y(n)_{\text{\'et}} \to (\Q/\Z)_Y(n)_{\text{\'et}} \to \Z_Y(n)_{\text{\'et}}[1]
\]
shows that the group $H^{s}_{\rm L} (Y_E, \Z(n))$ is isomorphic to  $H^{s}_{\rm L} (Y_E, \Q(n))$ if $s<0$, and consequently a $\Q$-vector space in this case.  Therefore, we have $E_2^{r,s} = H^r(G, H^{s}_{\rm L}(Y_E, \Z(n))) = 0$, if $s<0$ and $r>0$.  This implies that the filtration on $H^{r+s}_{\rm L} (Y, \Z(n))$ induced by the spectral sequence \eqref{motivic Hochschild-Serre} is always finite.  

Since the spectral sequence \eqref{motivic Hochschild-Serre} is functorial and compatible with products, the action of an element of $\End_k(X)$ respects the filtration induced by the Hochschild-Serre spectral sequence. If $\alpha \in \End_k(X)$ is such that $\alpha_E = 0$, then $\alpha_E$ acts by the zero map on each of the motivic cohomology groups $H^{s}_{\rm L} (Y_E, \Z(n))$ and hence, on every term $E_2^{r,s}$ of \eqref{motivic Hochschild-Serre}.  Consequently, the action of $\alpha$ on the Lichtenbaum cohomology groups $H^{m}_{\rm L} (Y, \Z(n))$ is nilpotent.
\end{proof}

\begin{proof}[Proof in the case $\Char(k)= p > 0$]
\hspace{1cm} 

We have an exact triangle 
\[
\Z[1/p]_Y(n)_{\text{\'et}} \to \Q_Y(n)_{\text{\'et}} \to (\Q/\Z)'_Y(n)_{\text{\'et}} \to \Z[1/p]_Y(n)_{\text{\'et}}[1]
\]
in the derived category of \'etale sheaves on $Y$, where $(\Q/\Z)'_Y(n)_{\text{\'et}} = \underset{\ell \neq p}{\oplus}~ \Q_{\ell}/\Z_{\ell}(n)$, where $\ell$ runs through all the primes except $p$.  An argument analogous to the one used in the characteristic $0$ case shows that we have a convergent Hochschild-Serre spectral sequence after inverting the exponential characteristic
\begin{equation}
\label{motivic Hochschild-Serre char p}
E_2^{r,s} = H^r (G, H^{s}_{\rm L} (Y_E, \Z[1/p](n))) \Rightarrow H^{r+s}_{\rm L} (Y, \Z[1/p](n)),
\end{equation}
for a Galois extension $E/k$ with Galois group $G$.  One can now follow the proof in the characteristic $0$ case step-by-step to show that the action of $\alpha \in \End_k(X)$ with $\alpha_E = 0$ is nilpotent on the groups $H^m_{\rm L} (Y, \Z[1/p](n))$, for every $m,n \in \Z$.
\end{proof}

\begin{remark} 
\label{remark nilpotence index}
The proof shows that for every such correspondence $\alpha$ with $\alpha_E=0$ for a Galois extension $E$ of $k$, the index of nilpotence $N$ of the action of $\alpha$ on $H^{m}_{\rm L} (Y, \Z(n))$ can be taken to be ${\rm min}\{{\rm cd}(k), m\} + 1$, where ${\rm cd}(k)$ is the cohomological dimension of the field $k$.  In particular, the index of nilpotence does not depend on $\alpha$.
\end{remark}

\section{Applications to Rost nilpotence}
\label{section applications to Rost nilpotence}

In this section, we use Theorem \ref{maintheorem} and the Bloch-Kato conjecture 
(proved by Rost-Voevodsky \cite{Voevodsky-Bloch-Kato}) to obtain a reformulation of the Rost nilpotence principle.  We then use this to study Rost nilpotence for schemes of dimension $\leq 3$. 

Let $Y$ be a smooth projective scheme over a field $k$ of characteristic $0$. 
If $\sF$ is a sheaf on $Y_{\text{\'et}}$, we let $\sH^{p}_{\text{\'et}}(\sF)_Y$ the Zariski sheaf associated with the presheaf $U \mapsto H_{\text{\'et}}^p(U, \sF)$ on the Zariski site $Y_{\text{Zar}}$.  We will abuse notation and write $\sH^{p}_{\text{\'et}}(\sF)$ for $\sH^{p}_{\text{\'et}}(\sF)_Y$, whenever there is no confusion.  

Let $\pi: Y_{\text{\'et}}\rightarrow Y_{\text{Zar}}$ be the canonical morphism of sites.  In the derived category 
of complexes of Zariski sheaves on $Y$, we have $\Z_Y(d) \xrightarrow{\sim} 
\tau_{\leq d+1} R\pi_* \Z_Y(d)_{\text{\'et}}$ (\cite{Suslin-Voevodsky}, \cite[Theorem 6.6]{Voevodsky-Milnor-conjecture}, 
\cite{Geisser-Levine-Crelle}; also see \cite[2D]{Kahn}), and hence the distinguished triangle
\begin{equation}
\label{equation Voevodsky exact triangle}
\Z_Y(d) \to R \pi_* \Z_Y(d)_{\text{\'et}} \to \tau_{\geq d+2} R \pi_* \Z_Y(d)_{\text{\'et}} \to \Z_Y(d)[1].
\end{equation}
\noindent The associated long exact sequence of Zariski hypercohomology groups yields the exact sequence 
\begin{equation}
\label{equation exact sequence B-L}
\CH^d_{\rm L}(Y,1) \to \HH^{2d-1}(Y, \tau_{\geq d+2} R \pi_* \Z_Y(d)_{\text{\'et}}) \to \CH^d(Y) \to \CH^d_{\rm L}(Y). 
\end{equation}
The group $\HH^{2d-1}(Y, \tau_{\geq d+2} R \pi_* \Z_Y(d)_{\text{\'et}})$ is the abutment of the hypercohomology spectral sequence
\begin{equation}
\label{equation hypercohomology spectral sequence}
E_2^{p,q} = H^p(Y, R^q \tau_{\geq d+2} R \pi_* \Z_Y(d)_{\text{\'et}}) \Rightarrow 
\HH^{p+q}(Y, \tau_{\geq d+2} R \pi_* \Z_Y(d)_{\text{\'et}}).
\end{equation}
It follows from  the quasi-isomorphism $\Z_Y(d) \xrightarrow{\sim} \tau_{\leq d+1} R \pi_* \Z_Y(d)_{\text{\'et}}$ together with \cite[Corollaire 2.8]{Kahn}, that 
\begin{equation}
\label{equation cohomology of truncation}
R^q \tau_{\geq d+2} R \pi_* \Z_Y(d)_{\text{\'et}} = 
\begin{cases}
0, \quad \quad \quad \quad \quad \ \ \ \ \text{if $q \leq d+1$};\\
\sH^{q-1}_{\text{\'et}}(\Q/\Z(d)), ~ \text{ if $q \geq d+2$}.
\end{cases}
\end{equation}
In particular, the $E_2$-terms of (\ref{equation hypercohomology spectral sequence}) 
are either trivial, or can be identified with the cohomology groups 
$H^p_{\text{Zar}}(Y, \sH^{q-1}_{\text{\'et}}(\Q/\Z(d)))$.  This facilitates the study of schemes of dimension 
$\leq 3$, as far as Rost nilpotence is concerned.  We begin by observing that this yields a simple proof of 
Rost nilpotence for surfaces in characteristic $0$, generalizing \cite[Theorem 9]{Gille-Surfaces}.

\begin{theorem}
\label{theorem Rost nilpotence dim at most 2}
Let $X$ be a smooth, projective scheme of dimension $\leq 2$ over a field $k$ of characteristic $0$.  Suppose that $\alpha \in \End_k(X)$ is such that $\alpha_E=0$ for a Galois extension $E/k$. Then $\alpha^{\circ N} = 0$, for some positive integer $N$, that is, $\alpha$ is nilpotent as a correspondence.  
\end{theorem}
\begin{proof}
Set $d=\dim X$ and $Y= X \times_k X$.  We first assume that $E/k$ is 
a Galois extension.  The case $d=0$ is trivial, and the case $d=1$ follows from Theorem 
\ref{maintheorem}, since $\CH^1(Y) \cong \CH_{\rm L}^1(Y)$, because of the 
quasi-isomorphism $\Z(1)_{\text{\'et}} \simeq \G_m[-1]$.  If $d=2$, then we have from  
\eqref{equation hypercohomology spectral sequence} and \eqref{equation cohomology of truncation}, 
the vanishing 
\[
\HH^{3}(Y, \tau_{\geq 4} R \pi_* \Z_Y(2)_{\text{\'et}}) = 0.
\]
Hence, the canonical comparison map $\CH^2(Y) \rightarrow \CH_{\rm L}^2(Y)$ is injective and the 
claim now follows from Theorem \ref{maintheorem}.

By a standard argument, it suffices to consider this case. Explicitly, if $E/k$ is an arbitrary field extension, we can find a tower of field extensions $k \subset F \subset E$, where $F/k$ is purely transcendental and $E/F$ is algebraic.  Since $F/k$ is purely transcendental, the restriction map 
\[
\res_{F/k}:\End_k(X) \to \End_F(X_F)
\]
\noindent is an isomorphism (see \cite[Proposition 2.1.8]{Flenner-OCarrol-Vogel}, for example).  Thus, 
in order to show that $\alpha$ is nilpotent, it suffices to show that $\alpha_F \in \End_F(X_F)$ is 
nilpotent.  Since $\Char (k) = 0$, we can find a tower of fields $F \subset E \subset E'$ with $E'/F$ 
Galois.  Since $\alpha_E = 0$, we have $\alpha_{E'} =0$, hence $\alpha$ is nilpotent.  
\end{proof}

\begin{remark}
\label{remark Gille comparison}
The proof Theorem \ref{maintheorem} shows that the index of nilpotence $N$ in Theorem \ref{theorem Rost nilpotence dim at most 2} can be taken to be $\min \{ cd(k), 2\dim_k(X)\} +1$. This improves the bounds on the nilpotence exponent obtained by Gille \cite[2.5, Corollary and Remark 11]{Gille-Surfaces}.  
\end{remark}

We next consider the case of smooth projective threefolds over a field $k$ with $\Char(k)=0$.  Since $X$ is a threefold, we have the identification 
$$
\HH^{5}(Y, \tau_{\geq 5} R \pi_* \Z_Y(3)_{\text{\'et}}) = H^{0}_{\rm Zar}(Y, \sH^{4}_{\text{\'et}}(\Q/\Z(3)))=H^{4}_{\rm nr}(Y, \Q/\Z(3)),$$
where $H^{4}_{\rm nr}(Y, \Q/\Z(3))$ is the unramified cohomology group of $Y= X \times_k X$.  Thus by \eqref{equation exact sequence B-L}, we have an exact sequence
\begin{equation}
\label{3fold-exact}
\CH^3(Y,1) \to H_{\rm L}^5(Y, \Z(3)) \to  H^{4}_{\rm nr}(Y, \Q/\Z(3)) \to \CH^3(Y) \to \CH^3_{\rm L}(Y).
\end{equation}

We wish to study the action of $\End_k(X)$ on the exact sequence \eqref{3fold-exact}.  The motivic cohomology groups can be identified as
\[
H^m_{\rm M}(Y, \Z(n)) = \Hom_{{\rm DM}^{\rm eff}(k, \Z)}(\Z(X), \Z(n)[m]),
\]
where ${\rm DM}^{\rm eff}(k, \Z)$ is the triangulated category of effective motives in the sense of \cite{Voevodsky-Suslin-Friedlander}.  Note that we have a compatible action of any $\alpha \in \End_k(X)$ on the complexes $\Z(Y)$ and $R\pi_*\pi^*\Z(Y)$ defined by the same formula as in \eqref{End-action} as the maps involved are all defined at the level of complexes.  We therefore have an induced compatible action of $\alpha$ on the cone of the comparison map $\Z(Y) \to R \pi_*\pi^* \Z(Y)$.  The long exact sequence \eqref{equation exact sequence B-L} is obtained by applying $\Hom_{{\rm DM}^{\rm eff}(k, \Z)}(-, \Z(n)[m])$ to the exact triangle in $DM^{\rm eff}(k, \Z)$ arising from the above comparison map.  Consequently, we get a compatible  $\End_k(X)$-action on \eqref{3fold-exact}.  Along with the exact sequence (\ref{3fold-exact}) and Theorem \ref{maintheorem}, this immediately implies the following criterion.

\begin{lemma}
\label{lemma threefold reformulation}
Let $X$ be a smooth projective integral threefold over a field $k$ with $\Char(k)=0$ and set $Y = X\times_k X$.  Assume
$\alpha \in \End_k(X)$ is such that $\alpha_E=0$ for a Galois extension $E/k$.  Then the action of $\alpha$ is nilpotent on $H^{4}_{\rm nr}(Y, \Q/\Z(3))$ if and only if it is nilpotent on $\CH^3(Y)$.
\end{lemma}

Thus, if $\Char (k)=0$, then the argument in the proof of Theorem \ref{theorem Rost nilpotence dim at most 2} shows that Rost nilpotence holds for $X$ if and only if the criterion in Lemma \ref{lemma threefold reformulation} is satisfied.

\begin{remark}
\label{3fold-rational} 
It is easy to see that the criterion from Lemma \ref{lemma threefold reformulation} applies to rational threefolds. Indeed, we may assume that $E/k$ is finite. Since $X$ is rational, so is $Y = X \times_k X$, and 
$H^{4}_{\rm nr}(Y, \Q/\Z(3)) \simeq H^{4}_{\text{\'et}}(k, \Q/\Z(3)))$.  Since $\alpha_E = 0$, the 
action of $\alpha$ on $$H^{q}_{\text{\'et}}(E, \Q/\Z(3))) \simeq 
H^{0}_{\rm Zar}(Y_E, \sH^{q}_{\text{\'et}}(\Q/\Z(3)))$$ is trivial, and the claim follows now from 
the Hochschild-Serre spectral sequence
\[
H^p(\Gal(E/k), H^{q}_{\text{\'et}}(E, \Q/\Z(3)))) \Rightarrow H^{p+q}_{\text{\'et}}(k, \Q/\Z(3))).
\]
In particular, if $\Char (k) = 0$, then Rost nilpotence holds for $X$.
\end{remark}

We now show how Lemma \ref{lemma threefold reformulation} can be used to prove Rost nilpotence for birationally ruled threefolds over a field of characteristic $0$.  Recall that a threefold is said to be birationally ruled if it is birational to $S \times \P^1$, where $S$ is a surface.  To this end, we show first that the Rost nilpotence principle is a birational invariant property of threefolds.  

\begin{lemma}
\label{lemma threefolds birational invariant}
Let $X$ be a smooth projective threefold over an arbitrary field $k$ and let $\phi:\~X \to X$ 
be the blow-up of $X$ with smooth center $C$.  Then $X$ satisfies the Rost nilpotence principle if and only if $\~X$ does.
\end{lemma}
\begin{proof}
If $C$ has pure codimension $r$ in $X$, the motive of $\~X$ is given by the formula 
\[
h(\~X) \simeq h(X) \oplus \left( \underset{i=1}{\overset{r-1} \oplus} h(C)(i) \right),
\]
see 
\cite[Section 9]{Manin}.
Since $\Dim X =3$, we only have to consider the cases when $\Dim C$ is $0$ or $1$.  When $\Dim C \leq 1$, it is easy to see that $\End_k(h(C))$ injects into $\End_E(h(C_E))$, for every field extension $E/k$.  Hence, we can apply \cite[Lemma, page 6]{Gille-geometrically-rational-surfaces} to complete the proof.
\end{proof}

\begin{theorem}
\label{theorem ruled threefold}
Let $X$ be a smooth projective birationally ruled threefold over a field $k$ of characteristic $0$.  Then $X$ satisfies the Rost nilpotence principle.
\end{theorem}
\begin{proof}
Let $X=S \times_k \P^1$, where $S$ is a surface. We show first that $X$ satisfies the Rost nilpotence principle.
Assume $\alpha \in \End_k(X)$ is such that $\alpha_E = 0$ for a finite Galois extension $E/k$. 
By Lemma \ref{lemma threefold reformulation} it suffices to show the action of $\alpha$ on 
$H^4_{\rm nr}(X\times_k X,\Q/\Z(3))$ is nilpotent. Since $k(X \times_k X)$ is a purely transcendental extension of $k(S \times_k S)$ of transcendence degree $2$, we have (by
\cite[Proposition 5.1]{Blinstein-Merkurjev}, for example) 
$
H^{4}_{\rm nr}(X \times_k X, \Q/\Z(3)) = H^{4}_{\text{\'et}}(k(S \times_k S), \Q/\Z(3)).
$
The group on the right hand side is the abutment of the spectral sequence
\[
E_2^{p,q} = H^p(\Gal(E/k), H^{q}_{\text{\'et}}(E(S \times_k S), \Q/\Z(3))) 
\Rightarrow H^{p+q}_{\text{\'et}}(k(S \times_k S), \Q/\Z(3))).
\]
\noindent 
For every $q$, we have $H^{q}_{\text{\'et}}(E(S \times_k S), \Q/\Z(3)) 
= H^{q}_{\rm nr}((X \times_k X)_E, \Q/\Z(3))$.  Since $\alpha_E = 0$, it follows that the action of $\alpha$ on $H^{q}_{\rm nr}((X \times_k X)_E, \Q/\Z(3))$ is zero.  Hence, the action of $\alpha$ on every $E_2^{p,q}$-term of the above spectral sequence is zero.  Consequently, the action of $\alpha$ on $H^{4}_{\rm nr}(X \times_k X, \Q/\Z(3))$ is nilpotent. If $E/k$ is an arbitrary field extension, the claim follows from an argument similar to the one in the proof of Theorem \ref{theorem Rost nilpotence dim at most 2}.

Now, let $X$ be a smooth projective birationally ruled threefold over $k$.  There exists a smooth 
projective surface $S$ such that $X$ is birational to $S \times_k \P^1$.  Since $\Char (k) = 0$, by 
\cite[Theorem 0.1.1]{AKMW} there exists a sequence of smooth projective varieties 
$Z_1, \ldots, Z_n, X_1, \ldots, X_n$ and a diagram
\[
\xymatrixcolsep{1.3pc}
\xymatrix{
  & Z_1 \ar[dl] \ar[dr] &     & Z_2 \ar[dl] \ar[dr] &&&&   Z_n \ar[dl] \ar[dr] & & \\
X &     & X_1 &     & {X_2} & \cdots          & X_n & & S \times_k \P^1}
\]
\noindent in which every morphism is a blow-up with a smooth center.  Applying Lemma 
\ref{lemma threefolds birational invariant}, we conclude that $X$ satisfies the Rost nilpotence principle.
\end{proof}

\begin{remark}
The discussion preceding Theorem \ref{theorem Rost nilpotence dim at most 2} suggests an approach to 
prove Rost nilpotence for schemes $X$ of dimension $d \geq 3$, which we briefly 
outline.  Assume $\alpha \in \End_k(X)$ is such that $\alpha_E=0$ for a Galois extension $E/k$.  
One may attempt to prove Rost nilpotence for $X$ by showing that for all 
$d+1 \leq p \leq 2d-2$ the action of the correspondence $\alpha$ is nilpotent on the cohomology groups $H^{2d-2-p}_{\rm Zar}(X \times_k X, \sH^{p}_{\text{\'et}}(\Q/\Z(d)))$ 
to obtain by \eqref{equation hypercohomology spectral sequence} a nilpotent action of $\alpha$ on the hypercohomology group 
$\HH^{2d-1}(X \times_k X, \tau_{\geq d+2} R \pi_* \Z(d)_{\text{\'et}})$, compatible with the action of $\alpha$ on $\CH^d_{\rm L}(X \times_k X,1)$ and $\CH^d(X \times_k X)$.  In view of Theorem \ref{maintheorem}, this would imply that the action of $\alpha$ on $\CH^d(X \times_k X)$ is nilpotent.
\end{remark}


\begin{thebibliography}{9999}

\bibitem{AKMW}
D. Abramovich, K. Karu, K. Matsuki, J. W{\l}odarczyk,
\emph{Torification and factorization of birational maps},
J. Amer. Math. Soc. 15(3) (2002) 531 -- 572.

\bibitem{Blinstein-Merkurjev}
S. Blinstein, A. Merkurjev,
\emph{Cohomological invariants of algebraic tori},
Algebra Number Theory 7 (2013), no. 7, 1643 -- 1684.

\bibitem{Bloch rational surfaces}
S. Bloch
\emph{On the Chow groups of certain rational surfaces},
Ann. Sci. \'Ecole Norm. Sup.(4), 14(1981), 41 -- 59.

\bibitem{Bloch Higher Chow groups}
S. Bloch,
\emph{Algebraic cycles and higher K-theory},
Adv. Math., 61(1986) 267 -- 304.

\bibitem{Bloch-Moving}
S. Bloch,
\emph{The moving lemma for higher Chow groups},
J. Algebraic Geom. 3 (1994), no. 3, 537 -- 568.

\bibitem{Brosnan}
P. Brosnan,
\emph{A short proof of Rost nilpotence via refined correspondences},
Doc. Math. 8(2003), 69 -- 78.

\bibitem{Chernousov-Gille-Merkurjev}
V. Chernousov, S. Gille, A. Merkurjev,
\emph{Motivic decomposition of isotropic projective homogeneous varieties},
Duke Math. J., 126(1)(2005), 137 -- 159.

\bibitem{Cisinski-Deglise}
D.-C. Cisinski, F. D\'eglise,
\emph{\'Etale motives},
Compositio Math. 152 (2016), no. 3, 556 -- 666.

\bibitem{Colliot-Thelene}
J.-L. Colliot-Th\'el\`ene,
\emph{Hilbert's Theorem $90$ for $K_2$, with application to the Chow groups of rational surfaces},
Invent. Math. 71 (1983), 1 -- 20.

\bibitem{CT-Kahn}
J.-L. Colliot-Th\'el\`ene, B. Kahn,
\emph{Cycles de codimension $2$ et $H^3$ non ramifi\'e pour les vari\'et\'es sur les corps finis},
J. K-Theory 11 (2013), no. 1, 1 -- 53.

\bibitem{Colliot-Thelene-Raskind}
J.-L. Colliot-Th\'el\`ene, W. Raskind,
\emph{$\sK_2$-cohomology and the second Chow group},
Math. Ann. 270(1985), 165 -- 199.

\bibitem{Deglise-RR}
F. D\'eglise,
\emph{Orientation theory in arithmetic geometry},
preprint, arXiv:1111.4203v2 [math.AG]

\bibitem{Elman-Karpenko-Merkurjev}
R. Elman, N. Karpenko, A. Merkurjev, 
\emph{The Algebraic and Geometric Theory of Quadratic Forms}, 
AMS Colloquium Publications, Vol. 56, 2008.

\bibitem{Flenner-OCarrol-Vogel}
H. Flenner, L. O'Carrol, W. Vogel
\emph{Joins and intersections},
Springer-Verlag, 1999.

\bibitem{Fulton}
W. Fulton,
\emph{Intersection theory},
Springer-Verlag, 1998.

\bibitem{Geisser-Levine-Invent}
T. Geisser, M. Levine,
\emph{The $K$-theory of fields in characteristic $p$},
Invent. Math. 139 (2000), 459 -- 493.

\bibitem{Geisser-Levine-Crelle}
T. Geisser, M. Levine,
\emph{The Bloch-Kato conjecture and a theorem of Suslin-Voevodsky},
J. Reine Angew. Math. 530(2001), 55 -- 103.

\bibitem{Gille-geometrically-rational-surfaces}
S. Gille,
\emph{The Rost nilpotence theorem for geometrically rational surfaces},
Invent. Math. 181(2010), no. 1, 1 -- 19.

\bibitem{Gille-Surfaces}
S. Gille,
\emph{On Chow motives of surfaces},
J. Reine Angew. Math. 686(2014), 149 -- 166.

\bibitem{Illusie}
L. Illusie,
\emph{Complexe de de\thinspace Rham-Witt et cohomologie cristalline},
Ann. Sci. \'Ecole Norm. Sup. (4), 12 (1979), no. 4, 501–661.

\bibitem{Kahn}
B. Kahn,
\emph{Classes de cycles motiviques \'etales},
Algebra Number Theory 6 (2012), no. 7, 1369 -- 1407.

\bibitem{Manin}
Y. Manin,
\emph{Correspondences, motifs and monoidal transformations},
(Russian) Mat. Sb. (N.S.) 77(119), 1968, 475 -- 507.

\bibitem{Milne logarithmic}
J. Milne,
\emph{Duality in the flat cohomology of a surface},
Ann. Sci. \'Ecole Norm. Sup. (4), 9 (1976), no. 2, 171 -- 201.

\bibitem{Rost-lemma}
M. Rost,
\emph{The motive of a Pfister form},
preprint, 1998.

\bibitem{Rost}
M. Rost,
\emph{Chow groups with coefficients},
Doc. Math. 1(1996), 316 -- 393.

\bibitem{Scholl}
A. Scholl,
\emph{Classical motives},
Motives (Seattle, WA, 1991), 163 -- 187, 
Proc. Sympos. Pure Math., 55, Part 1, American Mathematical Society, 1994.

\bibitem{Suslin}
A. Suslin,
\emph{Torsion in $K_2$ of fields},
$K$-theory, 1(1987), no. 1, 5 -- 29.

\bibitem{Suslin-Voevodsky}
A. Suslin, V. Voevodsky,
{\emph{Bloch-{K}ato conjecture and motivic cohomology with finite
coefficients}}, in: {The arithmetic and geometry of algebraic cycles 
({B}anff, {AB}, 1998)},
{NATO Sci. Ser. C Math. Phys. Sci.},
{548}, 117--189. 

\bibitem{Voevodsky-Milnor-conjecture}
V. Voevodsky,
\emph{Motivic cohomology with $\mathbb Z/2$-coefficients},
Publ. Math. Inst. Hautes \'Etudes Sci., 98(2003), 59 -- 104.

\bibitem{Voevodsky-Bloch-Kato}
V. Voevodsky,
\emph{On motivic cohomology with $\mathbb Z/\ell$-coefficients},
Ann. of Math. 174(2011), 401 -- 438.

\bibitem{Voevodsky-Suslin-Friedlander}
V. Voevodsky, A. Suslin, E. M. Friedlander, 
\emph{Cycles, transfers, and motivic homology theories}, 
Annals of Mathematics Studies, vol. 143, Princeton University Press, Princeton, NJ, 2000.

\bibitem{Weibel-Hbook}
C. Weibel,
\emph{An introduction to homological algebra},
Cambridge Studies in Advanced Mathematics, 38. Cambridge University Press, 1994.

\end{thebibliography}
\end{document}